\documentclass[12pt,a4paper]{amsart}

\usepackage{amsmath}
\usepackage{amsfonts}
\usepackage{amssymb}
\usepackage{graphicx}

\DeclareFontEncoding{FMS}{}{}
\DeclareFontSubstitution{FMS}{futm}{m}{n}
\DeclareFontEncoding{FMX}{}{}
\DeclareFontSubstitution{FMX}{futm}{m}{n}
\DeclareSymbolFont{fouriersymbols}{FMS}{futm}{m}{n}
\DeclareSymbolFont{fourierlargesymbols}{FMX}{futm}{m}{n}
\DeclareMathDelimiter{\VERT}{\mathord}{fouriersymbols}{152}{fourierlargesymbols}{147}

\usepackage{tikz}
\usetikzlibrary{matrix,arrows}
\usepackage{mathtools}

\newtheorem{theorem}{Theorem}[]
\newtheorem{proposition}{Proposition}[section]

\newtheorem{lemma}[proposition]{Lemma}
\theoremstyle{definition}

\newtheorem*{acknowledgements}{Acknowledgements}

\newcommand{\R}{\mathbb{R}} 
\newcommand{\C}{\mathbb{C}} 

\newcommand{\nn}{\nonumber} 

\DeclareMathOperator{\supp}{supp}

\author[]{Pedro Caro \and Kaloyan Marinov}
\title[]{Stability of inverse problems in an infinite Slab with partial data}
\date{}
\address{ICMAT - CSIC, Spain}
\email{pedro.caro@icmat.es}
\address{Department of Mathematics, University of Washington, USA}
\email{kmarinov@math.washington.edu}
\begin{document}

\begin{abstract}
In this paper, we study the stability of two inverse boundary value problems in an infinite slab with partial data. These problems have been studied by Li and Uhlmann in \cite{LU} for the case of the Schr\"odinger equation and by Krupchyk, Lassas and Uhlmann in \cite{KLU} for the case of the magnetic Scr\"odinger equation. Here we quantify the method of uniqueness proposed by Li and Uhlmann and prove a log-log stability estimate for the inverse problems associated to the Schr\"odinger equation. The boundary measurements considered in these problems are modelled by partial knowledge of the Dirichlet-to-Neumann map: in the first inverse problem, the corresponding Dirichlet and Neumann data are known on different boundary hyperplanes of the slab; in the second inverse problem, they are known on the same boundary hyperplane of the slab.
\end{abstract}

\maketitle


\section{Introduction}\label{sec: Intro}
This paper is devoted to the study of an inverse boundary value problem (IBVP) for the Schr\"odinger equation in an infinite slab. The problem consists of recovering the electric potential $q$ in the slab
\[ \Sigma := \{ x \in \R^3 : 0 < x_3 < L \},\]
from partial knowledge of the Dirichlet-to-Neumann map (DN map). Here, $L > 0$ is a constant, $x_3$ denotes the 3rd coordinate of $x$ and $q$ is compactly supported in
\[Q := \{ (x', x_3) \in \R^3: |x'| \leq R,\, 0 \leq x_3 \leq L \}\]
with $R > 0$ a constant. The DN map is roughly defined by
\[\Lambda_q : f \longmapsto \partial_\nu u|_{\partial \Sigma},\]
where $\partial \Sigma$ denotes the boundary of $\Sigma$, $\nu$ represents the outward-pointing unit normal vector along $\partial \Sigma$, $\partial_\nu = \nu \cdot \nabla$ and $u$ solves the problem\footnote{This problem is well-posed under certain conditions on $f$, $k$ and $q$ but, for the sake of simplicity, we omit details at this point.}
\begin{equation*}
\left\{ \begin{aligned}
			(-\Delta - k^2 + q)u &= 0 \,\text{in}\, \Sigma \\
			u|_{\partial \Sigma} &= f.
		\end{aligned} \right.
\end{equation*}
In \cite{LU}, Li and Uhlmann proved two uniqueness results for the potential $q$; each result assumes a different kind of partial knowledge of the DN map. In order to precisely describe these uniqueness results, we need to introduce some notation. The boundary of $\Sigma$ consists of the two hyperplanes
\[\Gamma_1 := \{ x \in \R^3 : x_3 = L \},\qquad \Gamma_2 := \{ x \in \R^3 : x_3 = 0 \}. \]
Choose $R' >0$ with $R < R'$, and set
\[ \Gamma_j^N := \{ x \in \Gamma_j : |x'| < R' \}, \quad j=1,2. \]
Let $\Gamma_1^D$ be a relatively open, precompact subset of $\Gamma_1$ such that
\[\overline{\Gamma_1^N} \subset \Gamma_1^D.\]
Let $q_1$ and $q_2$ be potentials from $L^\infty (\Sigma)$ such that both are (compactly) supported in $Q$, and let $\Lambda_{q_1}$ and $\Lambda_{q_2}$ denote their corresponding DN maps. Li and Uhlamann showed that if either
\[\Lambda_{q_1} f|_{\Gamma^N_1} = \Lambda_{q_2} f|_{\Gamma^N_1} \]
for all $f$ supported in $\overline{\Gamma^D_1}$, or
\[\Lambda_{q_1} f|_{\Gamma^N_2} = \Lambda_{q_2} f|_{\Gamma^N_2} \]
for all $f$ supported in $\overline{\Gamma^D_1}$, then
\[q_1 = q_2.\]
These results were extended by Krupchyk, Lassas and Uhlmann in \cite{KLU} to the case of the magnetic Scr\"odinger equation. In a slightly different situation (see \cite{Po}), Pohjola has been able to relax the assumptions on the region where the boundary data is measured.

In the last fifteen years, IBVPs with partial data have attracted a lot of attention and nowadays there is a fairly long list of publications studying such problems. In \cite{BU}, Bukhgeim and Uhlmann established, in dimension $n \geq 3$, uniqueness results for the IBVPs associated to the Schr\"odinger equation and the conductivity equation in the setting where the Dirichlet data is given on the whole boundary but the Neumann data is given only on (roughly speaking) half of the boundary. This result was improved by Kenig, Sj\"ostrand and Uhlmann in \cite{KSU}. Stability estimates for these problems have been established in \cite{HW} for the Bukhgeim and Uhlmann's result and in \cite{CDFR} and \cite{CDFR2} for the Kenig \textit{et al}'s result. It is important to point out that, so far, the best known stability for these problems is of log-log type. A partial reconstruction procedure was proposed by Nachman and Street in \cite{NS}. Other related results are \cite{DFKSU}, \cite{Ch}, \cite{Tz}, \cite{Ch2}, \cite{Ch3}, \cite{ST} and \cite{ChST}. Another important result with partial data is \cite{I}, where Isakov proved, in dimension $n = 3$, uniqueness for IBVPs associated  to the Schr\"odinger equation and the conductivity equation with partial data. In his paper, Isakov assumed the boundary of the domain to be partially flat or spherical and the measurements to be taken on the complement of the flat or spherical part. Wang and Heck proved in \cite{HW2} that Isakov's method provides the optimal stability for this inverse problem, that is, of log type (see \cite{M} in connection with the optimality issue). Related results are \cite{CaOS}, \cite{Ca11}, \cite{KLU} and \cite{L}. Other interesting results for IBVPs with partial data are \cite{AU}, \cite{IYU} \cite{GT}, \cite{KS}, \cite{AK}, \cite{BJ} and \cite{Fa}.

The basic tools to deal with this kind of partial-data IBVPs are integration by parts to obtain Alessandrini formulas and the construction of appropriate complex geometric optics (CGOs). In \cite{BU}, Bukhgeim and Uhlmann used a Carleman estimate with boundary terms to control the part of the boundary where no measurements were taken and then stated a type of Alessandrini formula. On the other hand, in \cite{I}, Isakov used a reflection argument across the flat part of the domain's boundary to construct CGOs vanishing on that flat part. In \cite{LU}, Li and Uhlmann took advantage of the geometry of the slab to combine the ideas from \cite{BU} and \cite{I} to prove their uniqueness results.

The main results in this paper are quantitative versions of Li and Uhlmann's results and will be stated in Section \ref{sec:statement of main result}. They consist of log-log-type stability estimates for the IBVPs under consideration. In order to explain the reason for the extra log in our estimate, we will now sketch the main points in our proof for the case where the Dirichlet and Neumann data are measured on different hyperplanes.

Let $q_1$ and $q_2$ denote two potentials with compact support in $Q$, and let $\Lambda^2_{q_1}$ and $\Lambda^2_{q_2}$ be defined by
\[\Lambda^2_{q_1} f = \Lambda_{q_1} f|_{\Gamma^N_2}, \qquad \Lambda^2_{q_2} f = \Lambda_{q_2} f|_{\Gamma^N_2},\]
for all $f$ supported in $\overline{\Gamma^D_1}$. The first step in our approach is to prove an integral estimate in which
\[\left| \int_\Sigma (q_1 - q_2) u_1 u_2 \, dx \right| \]
is bounded by $\| \Lambda_{q_1}^2 - \Lambda_{q_2}^2 \|_\ast$ plus some controllable terms, for a large enough set of functions $u_1$ and $u_2$ solving the equations $(-\Delta -k^2 +q_1)u_1 = 0$ and $(-\Delta -k^2 +q_2)u_2 = 0$ in a bounded domain $\Omega \subset \Sigma$ satisfying
\[\{ x \in \Sigma: |x'| \leq R\} \subset \Omega.\]
In order to obtain this estimate, we require $u_1$ to vanish along $\Gamma_2 \cap \partial \Omega$. The second step in our approach is to construct an appropriate family of solutions to extract information from the integral estimate. This will be a family of CGOs depending on a large parameter $\tau$. In order to ensure that $u_1$ meets the requisite condition $u_1|_{\Gamma_2 \cap \partial \Omega} = 0$, we will use Isakov's reflection argument from \cite{I}. The third step is to insert the CGOs into the integral estimate, which enables us to estimate (from above) the Fourier transform of $q_1 - q_2$ at frequencies from
\[\{ \xi = (\xi', \xi_3) \in \R^3: |\xi| < r,\, |\xi'| > 1 \}\]
in terms of $\| \Lambda_{q_2}^2 - \Lambda_{q_1}^2 \|_\ast$ and the parameter $\tau$. The forth step consists of extending the set of frequencies, at which the Fourier transform of $q_1 - q_2$ is controlled, to all of $\{ \xi \in \R^3 : |\xi| < r \}$ . To do so, we proceed as Liang did in \cite{L}: we use that the Fourier transform of $q_1 - q_2$ is analytic and a result from \cite{Ibook}. Thus, we are able to control all the low frequencies in a ball of arbitrary radius. Finally, we follow the ideas proposed by Alessandrini in \cite{A} to control first $\| q_1 - q_2 \|_{H^{-1}(\R^3)}$ and then $\| q_1 - q_2 \|_{L^\infty(\Sigma)}$.

The ingredients to achieve the first step are a Carleman estimate with boundary terms (proved and used in \cite{BU} by Bukhgeim and Uhlmann), a quantified unique continuation property from a proper boundary subset (due to Phung, see \cite{P}), and a Runge-type approximation argument (performed by Li and Uhlmann in \cite{LU}). Let us point out that, the unique continuation from a proper boundary subset produces the extra log in our estimate. Furthermore, in order to be able to complete the proof of our first step, which requires utilizing the Runge-type argument (density in $L^2$ sense), we need to introduce a new operator norm $\| \centerdot \|_\ast$ to establish the stability of the IBVPs under consideration. The more standard operator norm requires the Runge-type argument to hold in a stronger sense than the $L^2$ one but this does not seem to be possible. The fact of introducing $\| \centerdot \|_\ast$ to establish stability of this problem is one of the novelties of our approach in comparison to the previous literature on stability for IBVPs with partial data.

The analytic unique continuation used in the fourth step does not produce any extra log since we are not enlarging the size of frequencies, we are just extending to low frequencies. This situation is different from \cite{HW}, \cite{CDFR}, \cite{CDFR2} and \cite{CS}.

The approach used in the case where the Dirichlet and Neumann data are measured on the same hyperplane is quite similar to this one. In that case, we use CGOs to construct $u_1$ and $u_2$ in a such a way that both of them vanish on $\Gamma_2 \cap \partial \Omega$; as a consequence, no Carleman estimate is required, so the proof of the integal estimate turns out to be simpler. However, the rest of the argument requires a quantification of the Riemann-Lebesgue lemma (cf. the proof of Theorem 8.22(f) from \cite{F} for functions in $C_c^\infty(\R^n)$).

The paper is organized as follows. In Section \ref{sec:statement of main result}, we state the main results of this article. In Section \ref{sec:integralEst}, we prove the integral estimates for the two IBVPs under consideration. In Section \ref{sec:theorem1}, we prove the stability of the problem when the Dirichlet and Neumann data are measured on different hyperplanes. Section \ref{sec:theorem2} is dedicated to the case where measurements are made on the same hyperplane.

\section{Main results}\label{sec:statement of main result}
In this section, we state the stability estimates that we announced in the introduction. In order to be precise, we will review some points from Section \ref{sec: Intro} with more details.

Let $K$ be an arbitrary compact subset of $\Gamma_1$, and define
\[H^{3/2}_K(\Gamma_1) := \{ f\in H^{3/2}(\Gamma_1) : \supp f \subseteq K \}.\]
Fix a potential $q \in L^\infty (\Sigma)$ which is compactly supported in $Q$. For a certain frequency $k \geq 0$ that we call admissible for $q$, we know that, given a compactly supported $w \in L^2(\Sigma)$, there exists a unique $v \in H^2_\mathrm{loc} (\overline{\Sigma})$ such that
\begin{equation}
\left\{ \begin{array}{r c l}
(-\Delta - k^2 + q) v &=& w \text{ in } \Sigma, \\
v|_{\partial \Sigma} &=& 0.
\end {array} \right. \label{pb:homoDIrichlet_q}
\end{equation}
Moreover, for any bounded subset $\Omega \subset \Sigma$, we have the estimate
\[\| v \|_{H^2(\Omega)} \leq C \| w \|_{L^2(\Sigma)},\]
where the constant $C > 0$ depends on $k, \Omega$, and any upper bound on $\| q \|_{L^\infty(\Sigma)}$. For an account of this direct problem and a discussion of admissible frequencies, see 
\cite{KLU}. The estimate bounding $v$ in $\Omega$ was not stated in \cite{KLU} but follows from their considerations.

The well-posedness of boundary value problem \eqref{pb:homoDIrichlet_q} implies that, given any $f \in H^{3/2}_K (\Gamma_1)$, there exists a unique admissible solution $u \in H^2_\mathrm{loc} (\overline{\Sigma})$ to the following Dirichlet problem
\begin{equation}
\left\{ \begin{array}{llll}
( -\Delta - k^2 + q) &u &= &0 \text{ in } \Sigma,  \\
\ &u|_{\Gamma_1} &= &f,  \\
\ &u|_{\Gamma_2} &= &0. 
\end{array} \right. \label{pb:Dirichlet_q}
\end{equation}
The well-posedness of this problem allows us to define the following DN map
\begin{eqnarray}
\Lambda_q: H^{3/2}_K(\Gamma_1) &\rightarrow& H^{1/2}_{\text{loc}}(\partial \Sigma). \nn \\
f &\mapsto& \partial_\nu u|_{\partial \Sigma} \nn
\end{eqnarray}
where $u$ is the unique admissible solution to the problem \eqref{pb:Dirichlet_q}. Let $\Lambda^1_q$ and $\Lambda^2_q$ denote the maps defined by
\begin{equation*}
\Lambda^1_q f := \Lambda_q f|_{\Gamma_1^N}, \quad \Lambda^2_q f:= \Lambda_q f|_{\Gamma_2^N}, \qquad \forall\, f \in H^{3/2}_{\overline{\Gamma^D_1}}(\Gamma_1).
\end{equation*}
Now we are ready to state the main results of this paper.
\begin{theorem} \label{th:op-norm} \sl
Let $k \geq 0$ be an admissible frequency for the zero potential. Then, there exists a norm $\VERT \centerdot \VERT$ on $H^{3/2}_{\overline{\Gamma^D_1}}(\Gamma_1)$, which depends on $k$, such that, if $q \in L^\infty (\Sigma)$ with $\supp q \subseteq Q$ and if $k$ is admissible for $q$, then $\Lambda^l_q$ is a bounded operator from $\left( H^{3/2}_{\overline{\Gamma^D_1}}(\Gamma_1), \VERT \centerdot \VERT \right)$ to $H^{-3/2}(\Gamma^N_l)$.
\end{theorem}
Let $\| \centerdot \|_\ast$ denote the operator norm of bounded linear operators from $\left( H^{3/2}_{\overline{\Gamma^D_1}}(\Gamma_1), \VERT \centerdot \VERT \right)$ to $H^{-3/2}(\Gamma^N_l)$.

\begin{theorem} \label{th:diffDdNd} \sl Consider $s > 3/2$, and let $q_1$, $q_2$ belong to $H^s (\Sigma)$ and have their supports contained in $Q$. Consider $k \geq 0$ to be admissible for $q_1$, $ q_2$ and the zero potential. Let $M$ denote an upper bound on $\| q_j \|_{H^s(\Sigma)} \leq M$. Then, there exists $\delta = \delta(L, R, k) > 0$ such that, if $\| \Lambda^2_{q_2} - \Lambda^2_{q_1} \|_\ast < 1/\delta$, then
\[\| q_1 - q_2 \|_{L^\infty (\Sigma)} \lesssim \big( \log [ 1 + | \log (\delta \| \Lambda^2_{q_2} - \Lambda^2_{q_1} \|_\ast) | ] \big)^{-\theta \frac{s - 3/2}{s + 1}} \]
with $0 < \theta < 1/10$.
The implicit constant\footnote{Throughout the paper, we will write $a\lesssim b$ whenever $a$ and $b$ are non-negative quantities that satisfy $a \leq C b$ for a certain constant $C > 0$. A constant $C > 0$ satisfying the previous inequality will be called an \textit{implicit constant} and it will only depend on unimportant quantities such as $L, R, k, M, s$ and $\delta$.} only depends on $L, R, k, M, s$ and $\delta$.
\end{theorem}

\begin{theorem} \label{th:sameDdNd} \sl
Consider $s > 3/2$, and let $q_1, q_2$ belong to $H^s (\R^3)$ and have their supports contained in $Q$. Consider $k$ to be admissible for $q_1$, $ q_2$ and the zero potential. Let $M$ denote an upper bound on $\| q_j \|_{H^s(\R^3)} \leq M$. Then, there exists $\delta = \delta(L, R, k)> 0$ such that, if $\| \Lambda^1_{q_2} - \Lambda^1_{q_1} \|_\ast < 1/\delta$, then
\[\| q_1 - q_2 \|_{L^\infty (\Sigma)} \lesssim \big( \log [ 1 + | \log (\delta \| \Lambda^1_{q_2} - \Lambda^1_{q_1} \|_\ast) | ] \big)^{- \theta \frac{s - 3/2}{s + 1}}. \]
with $0 < \theta < 1/5$.
The implicit constant in the last inequality depends on the same parameters as the implicit constant from the inequality in Theorem \ref{th:diffDdNd}.
\end{theorem}

Our results hold in dimension $n = 3$. We have only considered the three dimensional case for the sake of simplicity but we believe that these results also hold for $n > 3$ following similar arguments.

\section{Integral estimates}\label{sec:integralEst}
The main goal of this section is to prove the integral estimates that we announced in the introduction. Before stating these estimates, we will introduce a norm for $H^{3/2}_{\overline{\Gamma^D_1}}(\Gamma_1)$ and we will prove Theorem \ref{th:op-norm}.

Let $k\geq0$ be an admissible frequency for the zero potential in $\Sigma$; we define,  for each $f \in H^{3/2}_{\overline{\Gamma^D_1}} (\Gamma_1)$, the norm
\begin{equation}
\VERT f \VERT := \| v_f \|_{L^2(\Omega)}, \label{def:tripleNORM}
\end{equation}
where $v_f \in H^2_\mathrm{loc} (\overline{\Sigma}) $ is the unique solution to
\begin{eqnarray}
\left\{ \begin{array}{r c l}
-(\Delta + k^2) v_f &=& 0 \text{ in } \Sigma, \\
v_f|_{\Gamma_1} &=& f,  \\
v_f|_{\Gamma_2} &=& 0;
\end {array} \right. \label{pb:Dirichlet_0} 
\end{eqnarray}
and $\Omega$ is a bounded open subset of $\Sigma$ which satisfies
\[\{ x \in \Sigma : |x'| \leq R' \} \subset \Omega\]
and has a smooth boundary $\partial \Omega$ such that
\[\partial \Omega \cap \Gamma_1 \subseteq \Gamma_1^D, \qquad \overline{\Gamma_j^N} \subseteq \mathrm{int}_{\Gamma_j} (\partial \Omega \cap \Gamma_j) \]
for $j = 1, 2$. {Since we want $\Gamma_1^D$ and $\Gamma^N_j$ to be as small as possible, we now assume $R' < 2R$; at this moment, we fix $\Omega$ satisfying all of the above conditions together with}
\[\Omega \subset \{ x \in \Sigma : |x'| \leq 2 R \}.\]
The norm $\VERT \centerdot \VERT$ obviously depends on $\Omega$ and $k$ but these dependences are harmless for our problems. 
The well-posedness of the problem \eqref{pb:Dirichlet_0}, together with the fact that 
\begin{equation} 
\VERT f \VERT = 0  \Rightarrow f = 0, \label{proper:norm}
\end{equation}
guarantee that $\VERT \centerdot \VERT$ is a norm on $H^{3/2}_{\overline{\Gamma_1^D}} (\Gamma_1)$. The property \eqref{proper:norm} follows from the weak unique continuation property for the equation $-( \Delta + k^2) v_f = 0$ in $\Sigma$.

With this new norm on $H^{3/2}_{\overline{\Gamma_1^D}} (\Gamma_1)$, we will show that $\Lambda^j_q$ is a bounded operator.
\begin{lemma} \label{lem:op-norm} \sl
The following inequality holds
\[\| \Lambda^j_q  f\|_{H^{-3/2}(\Gamma^N_j)} \lesssim \VERT f \VERT,\]
for every $f \in H^{3/2}_{\overline{\Gamma^D_1}} (\Gamma_1)$, where 
\begin{equation}
\| \Lambda^j_q  f\|_{H^{-3/2}(\Gamma^N_j)} := \sup_{g \in H^{3/2}_{\overline{\Gamma^N_j}} (\Gamma_j) \setminus \{ 0 \}}  \frac{| \int_{\Gamma_j} \Lambda_q f\, g \,dx'|}{\| g  \|_{H^{3/2}(\Gamma_j)}}. \label{def:H-3/2}
\end{equation}
The implicit constant here depends on $k$, any upper bound on $\| q \|_{L^\infty(\Sigma)}$ and $\Omega$.
\end{lemma}
Note that Theorem \ref{th:op-norm} is an immediate consequence of this lemma. Moreover, Lemma \ref{lem:op-norm} still holds if each occurrence of $\overline{\Gamma^D_1}$ in its statement is replaced by any compact subset $K$ of $\Gamma_1$. In particular, the intersection between $\overline{\Omega}$ and $K$ is even allowed to be empty.

\begin{proof}
Fix $f \in H^{3/2}_{\overline{\Gamma^D_1}} (\Gamma_1)$. For any $g \in H^{3/2}_{\overline{\Gamma^N_j}} (\Gamma_j)$, we have that
\[ \int_{\Gamma_j} \Lambda_q f\, g \,dx' = \int_{\Gamma_j} \partial_\nu u\, g \,dx' \]
with $u$ solving \eqref{pb:Dirichlet_q}. By the trace theorem for $\Omega$, there exists $v \in H^2 (\Omega)$  such that $v(x) = g(x)$ for almost every $x \in \Gamma^N_j$, $v(x) = 0 $ for almost every $x \in \partial \Omega \setminus \Gamma^N_j$, $\partial_\eta v |_{\partial \Omega} = 0$ and
\begin{equation}
\| v \|_{H^2(\Omega)} \lesssim \| g \|_{H^{3/2}(\Gamma_j)}. \label{es:rightINV}
\end{equation}
Here $\eta$ denotes the outward-pointing unit normal vector along $\partial \Omega$, and the implicit constant depends on $\Omega$. Then, using Green's formula, we get that
\[ \int_{\Gamma_j} \Lambda_q f\, g \,dx' = \int_\Omega \Delta u v - u \Delta v \,dx \]
which, by \eqref{es:rightINV}, implies
\[\left| \int_{\Gamma_j} \Lambda_q f\, g \,dx' \right| \lesssim \big( \| u \|_{L^2(\Omega)} + \| \Delta u \|_{L^2(\Omega)} \big) \| g \|_{H^{3/2}(\Gamma_j)}. \]
Since $u$ is solution to \eqref{pb:Dirichlet_q}, we have
\[\| \Delta u \|_{L^2(\Omega)} \leq (k^2 + \| q \|_{L^\infty (\Sigma)}) \| u \|_{L^2(\Sigma)}\]
and therefore, by \eqref{def:H-3/2},
\[\| \Lambda^j_q  f\|_{H^{-3/2}(\Gamma^N_j)} \lesssim \| u \|_{L^2(\Omega)},\]
where the implicit constant depends on $k$, any upper bound on $\| q \|_{L^\infty(\Sigma)}$ and $\Omega$.

Let $w$ be defined by $w:= u - v_f$ with $v_f$ as in \eqref{pb:Dirichlet_0}. Then, $u = w + v_f$ with $w$ being the unique solution to the Dirichlet problem
\begin{eqnarray*}
\left\{ \begin{array}{r c l}
(-\Delta - k^2 + q) w &=& - q v_f \text{ in } \Sigma, \\
w|_{\partial \Sigma} &=& 0.
\end {array} \right.
\end{eqnarray*}
By the triangle inequality and the well-posedness of this problem, we deduce
\[\| \Lambda^j_q  f\|_{H^{-3/2}(\Gamma^N_j)} \lesssim \| v_f \|_{L^2(\Omega)},\]
which is nothing but the claimed inequality.
\end{proof}

Next, we turn our attention to the integral estimates, which can be stated as follows.

\begin{proposition} \label{prop:integralEst} \sl
Fix potentials $q_1, q_2 \in L^\infty(\Sigma)$ both of which are compactly supported in $Q$, and let $M>0$ denote an upper bound on $\|q_j \|_{L^\infty(\Sigma)} \leq M$ for $j=1,2$. Consider $k \geq 0$ to be admissible for $q_1, q_2$ and the zero potential. Assume that $u_1$ and $u_2$ belong to $H^2 (\Omega)$ and are solutions to
\begin{eqnarray}
(-\Delta - k^2 + q_1) u_1 &=& 0 \text{ in } \Omega, \nn \\
u_1|_{\Gamma_2 \cap \partial \Omega} &=& 0 \nn
\end{eqnarray}
and
\[
(-\Delta - k^2 + q_2) u_2 = 0 \text{ in } \Omega,
\]
respectively.
\begin{itemize}
\item[(a)] If $u_2|_{\Gamma_2 \cap \partial \Omega} = 0$, then there exists a constant $\delta = \delta(L, R, k) > 0$ such that, if $\| \Lambda^1_{q_2} - \Lambda^1_{q_1} \|_\ast < 1/\delta $, then
\begin{equation*}
\begin{aligned}
\bigg|\int_\Omega (q_1 - q_2) u_1 u_2 \, dx\bigg| \lesssim & \frac{\| u_1 \|_{L^2(\Omega)} \| u_2 \|_{H^2 (\Omega)}} {\left[ 1 + \bigg| \log \left(\delta \| \Lambda^1_{q_2} - \Lambda^1_{q_1} \|_\ast \right) \bigg| \right]^{1/2}}.
\end{aligned}
\end{equation*}

\item[(b)] There exist constants $C=C(L, R)>0$ and $\delta = \delta(L, R, k) > 0$ such that, if $\| \Lambda^2_{q_2} - \Lambda^2_{q_1} \|_\ast < 1/\delta $, then
\begin{equation*}
\begin{aligned}
\bigg|\int_\Omega (q_1 - q_2) u_1 u_2 \, dx\bigg| \lesssim & e^{c \tau |\zeta| } \frac{\| u_1 \|_{L^2(\Omega)} \| u_2 \|_{H^2 (\Omega)}} {\left[ 1 + \bigg| \log \left(\delta \| \Lambda^2_{q_2} - \Lambda^2_{q_1} \|_\ast \right) \bigg| \right]^{1/2}} \\
&+ \frac{1}{\tau^{1/2}} \| e^{\tau x \cdot \zeta} u_2 \|_{H^1 (\Omega)} \| e^{-\tau x \cdot \zeta} u_1 \|_{L^2(\Omega)}
\end{aligned}
\end{equation*}
for all $\tau \geq \tau_0 := C (k^2 + M)$ and $\zeta \in \R^3$ with $\zeta \cdot \eta|_{\Gamma^N_1} \geq 1 $; here, $c > 2(2R + L)$.
\end{itemize}
\end{proposition}

\begin{proof}
Let $v_1 \in H^2_\mathrm{loc} (\overline{\Sigma}) $ be a solution to $ (-\Delta - k^2 + q_1) v_1 = 0 $ in $ \Sigma $ with $ \supp (v_1|_{\partial \Sigma}) \subseteq \overline{\Gamma^D_1}$. Writing $f := v_1|_{\Gamma_1}$, we know that there exists a unique $v_2 \in H^2_\mathrm{loc} (\overline{\Sigma})$ such that
\begin{eqnarray}
(-\Delta - k^2 +q_2) v_2 &=& 0 \text{ in } \Sigma, \nn \\
v_2|_{\Gamma_1} &=& f, \nn \\
v_2|_{\Gamma_2} &=& 0. \nn
\end{eqnarray}
Then, $w:= v_2 - v_1$ belongs to $H^2_{\text{loc}}(\overline{\Sigma})$, and it is the unique admissible solution of
\begin{equation}
\begin{aligned}
(-\Delta - k^2 +q_2) w &= (q_1 - q_2) v_1 \text{ in } \Sigma, \\
w|_{\partial \Sigma} &= 0.
\end{aligned} \label{pb:w}
\end{equation}

Obviously,
\begin{equation}
\begin{aligned}
\int_\Omega (q_1 &- q_2) u_1 u_2 \, dx \\
&= \int_\Omega (q_1 - q_2) v_1 \chi u_2 \, dx + \int_\Omega (q_1 - q_2) (u_1 - v_1) u_2 \, dx \label{id:approx}
\end{aligned}
\end{equation}
where $\chi$ is a bump function in $\R^2$ which satisfies $\chi(x') = 1$ for $|x'| \leq R + \epsilon $ and $\supp \chi \subset \{ |x'| \leq R' -\epsilon \}$ for a small enough $\epsilon > 0$. Using the equation solved by $w$, applying Green's formula in $\Omega$, utilizing the equation satisfied by $u_2$ together with $w|_{\partial \Sigma} = 0$ and taking advantage of $\chi = 0$ in a neighbourhood of $\partial \Omega \cap \Sigma$, we get
\begin{equation}
\begin{aligned}
\int_\Omega (q_1 &- q_2) v_1 \chi u_2 \, dx \\
& = - \int_\Omega w ( \Delta\chi u_2 + 2 \nabla \chi \cdot \nabla u_2 ) \, dx - \int_{\Gamma^N_1 \cup \Gamma^N_2} \chi u_2 \partial_\nu w\, d x'.
\end{aligned} \label{id:integrationBYparts}
\end{equation}

Using \eqref{id:approx}, \eqref{id:integrationBYparts} and that $\supp q_j \subseteq Q$ for $j = 1, 2$, we immediately see that
\begin{equation}
\begin{aligned}
\bigg|\int_\Omega (&q_1 - q_2) u_1 u_2 \, dx\bigg| \lesssim \| \chi (u_1 - v_1) \|_{L^2(\Omega)} \| \chi u_2 \|_{L^2(\Omega)} \\
 &+ \bigg| \int_\Omega w ( \Delta\chi u_2 + 2 \nabla \chi \cdot \nabla u_2 ) \, dx \bigg| + \bigg| \int_{\Gamma^N_1 \cup \Gamma^N_2} \chi u_2 \partial_\nu w\, d x' \bigg|.
\end{aligned} \label{es:INTestimate}
\end{equation}

We next have to obtain an upper bound on each term in the previous inequality. The method for estimating each of the two boundary integrals depends on whether the domain of integration does or does not coincide with the part of $\partial \Sigma$ on which the Neumann data is measured. The method for estimating the interior integral on the right-hand side of \eqref{es:INTestimate} relies on a quantified unique continuation property for $w$.

To fix ideas, let the Neumann data be measured on $\Gamma_l^N$. Start by estimating the boundary integral along $\Gamma_l^N$ from \eqref{es:INTestimate}. Using $\partial_\nu w|_{\partial \Sigma} = (\Lambda_{q_2} - \Lambda_{q_1}) f$, $\supp\chi \subseteq \{ |x'| \leq R' \}$, and \eqref{def:H-3/2}, we get
\[\bigg| \int_{\Gamma^N_l} \chi u_2 \partial_\nu w\, d x' \bigg| \leq \| \chi u_2 \|_{H^{3/2} (\Gamma_l)} \| (\Lambda_{q_1}^l - \Lambda_{q_2}^l)f \|_{H^{-3/2}(\Gamma^N_l)}.\]
The last term on the right-hand side can be estimated using the definition of the operator norm and \eqref{def:tripleNORM} as follows:
\begin{equation}
\| (\Lambda_{q_1}^l - \Lambda_{q_2}^l)f \|_{H^{-3/2}(\Gamma^N_l)} \leq \| \Lambda_{q_1}^l - \Lambda_{q_2}^l \|_\ast (\| v_f - v_1 \|_{L^2 (\Omega)} + \| v_1 \|_{L^2(\Omega)}), \label{es:OPboundedness}
\end{equation}
where $v_f$ satisfies \eqref{pb:Dirichlet_0}. Note that $v_f - v_1$ satisfies
\begin{eqnarray*}
\left\{ \begin{array}{r c l}
(-\Delta - k^2) (v_f - v_1) &=&  q_1 v_1 \text{ in } \Sigma, \\
(v_f - v_1)|_{\partial \Sigma} &=& 0.
\end {array} \right.
\end{eqnarray*}
By the well-posedness of this problem, we have
\begin{equation}
\| v_f - v_1 \|_{L^2(\Omega)} \lesssim \| \chi v_1 \|_{L^2(\Sigma)}. \label{es:vf-u1}
\end{equation}
Thus, using \eqref{es:OPboundedness}, \eqref{es:vf-u1} and the boundedness of the trace operator associated with $\Omega$, the boundary term under consideration is bounded in the following way:
\begin{equation}
\bigg| \int_{\Gamma^N_l} \chi u_2 \partial_\nu w\, d x' \bigg| \lesssim \| \Lambda_{q_1}^l - \Lambda_{q_2}^l \|_\ast \| u_2 \|_{H^2 (\Omega)} \| v_1 \|_{L^2(\Omega)}. \label{ter:boundaryDIRECTLY}
\end{equation}
Under the assumptions in (a), the inequalities \eqref{ter:boundaryDIRECTLY} and \eqref{es:INTestimate} imply
\begin{equation}
\begin{aligned}
\bigg|\int_\Omega (q_1 &- q_2) u_1 u_2 \, dx\bigg| \lesssim \| u_1 - v_1 \|_{L^2(\Omega)} \| u_2 \|_{L^2(\Omega)} \\
 &+ \| w \|_{L^2(Q')} \| u_2 \|_{H^1(\Omega)} + \| \Lambda_{q_1}^1 - \Lambda_{q_2}^1 \|_\ast \| u_2 \|_{H^2 (\Omega)} \| v_1 \|_{L^2(\Omega)},\\
\end{aligned} \label{es:NGamma1}
\end{equation}
where $Q' := \{ x\in\Sigma : R + \epsilon < |x'| < R' - \epsilon \}$. In order to get the estimate in (a) from \eqref{es:NGamma1}, we have to control $w$ in $Q'$ and $u_1 - v_1$ in $\Omega$. We postpone this for a while; instead, we now focus on estimating the other boundary term in \eqref{es:INTestimate}, which only appears under the assumptions in (b). More concretely, we focus on estimating the term
\begin{equation}
\bigg| \int_{\Gamma^N_1} \chi u_2 \partial_\nu w\, d x' \bigg| \label{term:boundaryGamma2}
\end{equation}
in terms of $\| \Lambda_{q_2}^2 - \Lambda_{q_1}^2 \|_\ast$ and a sufficiently large parameter $\tau$.

To fix ideas, let $\zeta \in \R^3$ be arbitrarily chosen with $\zeta \cdot e_3 \geq 1$. In order to control \eqref{term:boundaryGamma2}, we use a Carleman inequality proven by Bukhgeim and Uhlmann in \cite{BU} (see Corollary 2.3).  Since $|\zeta| \geq |\zeta \cdot e_3| \geq 1$, the Carleman inequality can be applied to our situation as follows: For any $q \in L^\infty(\Omega)$ with $\| q \|_{L^\infty(\Omega)} \leq M$, there exists a constant $C = C(L, R) > 0$ such that
\begin{equation}
\begin{aligned}
\tau^2 \int_\Omega |e^{-\tau x \cdot \zeta} u|^2 \,dx &+ \tau \int_{\partial \Omega} (\zeta \cdot \eta) |e^{-\tau x \cdot \zeta} \partial_\eta u|^2 \, dS \\
&\lesssim \int_\Omega |e^{-\tau x \cdot \zeta} (-\Delta - k^2 + q)u|^2 \, dx
\end{aligned} \label{es:carleman}
\end{equation}
for all $u \in H^2(\Omega)$ with $u|_{\partial \Omega} = 0$, $\tau \geq C (k^2 + M)$; the implicit constant in \eqref{es:carleman} depends on $R$ and $L$.

Start by noting that
\begin{equation}
\bigg| \int_{\Gamma^N_1} \chi u_2 \partial_\nu w\, d x' \bigg| \leq \| e^{\tau x \cdot \zeta} u_2 \|_{L^2 (\Gamma^N_1)} \| e^{-\tau x \cdot \zeta} \partial_\eta (\chi w) \|_{L^2 (\Gamma^N_1)} \label{es:boundaryCARLE}
\end{equation}
since $\eta|_{\Gamma_1^N} = \nu|_{\Gamma^N_1}$ is a constant multiple of $e_3$ and since $\partial_{x_3} \chi = 0$. Here $e_3$ denotes the vector satisfying $x_3 = e_3 \cdot x$. The first term on the right-hand side can be bounded as follows
\begin{equation}
\| e^{\tau x \cdot \zeta} u_2 \|_{L^2 (\Gamma^N_1)} \lesssim \| e^{\tau x \cdot \zeta} u_2 \|_{H^1(\Omega)} \label{es:otraantes}
\end{equation}
using the boundedness of the trace operator associated with $\Omega$, where the implicit constant depends on $\Omega$. 
We estimate the second term on the right-hand side of \eqref{es:boundaryCARLE} as
\begin{equation}
\| e^{-\tau x \cdot \zeta} \partial_\eta (\chi w) \|^2_{L^2 (\Gamma^N_1)} \leq \int_{\Gamma^N_1} \zeta \cdot \eta |e^{-\tau x \cdot \zeta} \partial_\eta (\chi w)|^2 \, dx'. \label{es:unamas}
\end{equation}
Since $\chi w \in H^2(\Omega)$ vanishes on $\partial \Omega$, an application of \eqref{es:carleman} with $u$ replaced by $\chi w$ and $q$ replaced by $q_2$ shows that the right-hand side of \eqref{es:unamas} can be bounded by
\begin{equation}
\frac{1}{\tau} \int_\Omega |e^{-\tau x \cdot \zeta} (-\Delta - k^2 + q_2)(\chi w)|^2 \, dx + |\zeta| e^{2c\tau |\zeta|} \| \chi \partial_\nu w\|^2_{L^2(\Gamma_2)}, \label{term:outputCARLE}
\end{equation}
where $c := 2R + L$ is not the $c$ from the statement of Proposition \ref{prop:integralEst} (b). Furthermore, since $w$ solves \eqref{pb:w}, we have
\begin{align*}
\int_\Omega |e^{-\tau x \cdot \zeta}& (-\Delta - k^2 + q_2)(\chi w)|^2 \, dx \lesssim e^{2c\tau |\zeta|} \| v_1 - u_1 \|^2_{L^2(\Omega)} \\
&+ \| e^{-\tau x \cdot \zeta} u_1 \|^2_{L^2(\Omega)} + \int_\Omega |e^{-\tau x \cdot \zeta} (\Delta \chi w +2 \nabla\chi \cdot \nabla w) |^2 \, dx \\
\lesssim & e^{2c\tau|\zeta|} \| w \|^2_{H^1(Q')} + \big( e^{2c\tau |\zeta|} \| v_1 - u_1 \|^2_{L^2(\Omega)} + \| e^{-\tau x \cdot \zeta} u_1 \|^2_{L^2(\Omega)} \big).
\end{align*}
These computations are meant to bound the first term in \eqref{term:outputCARLE}. We now take care of the second one. By interpolation and using that $\partial_\nu w|_{\partial \Sigma} = (\Lambda_{q_2} - \Lambda_{q_1}) f$, we get
\[\| \chi \partial_\nu w\|_{L^2(\Gamma_2)} \leq \| \chi (\Lambda_{q_2} - \Lambda_{q_1}) f \|^{1/4}_{H^{-3/2}(\Gamma_2)} \| \chi \partial_\nu w\|^{3/4}_{H^{1/2}(\Gamma_2)}. \]
It is a simple computation to show that
\[\| \chi (\Lambda_{q_2} - \Lambda_{q_1}) f \|_{H^{-3/2}(\Gamma_2)} \lesssim \| (\Lambda^2_{q_2} - \Lambda^2_{q_1}) f \|_{H^{-3/2}(\Gamma^N_2)}\]
with the implicit constant depending on $R$. Following \eqref{es:OPboundedness} and \eqref{es:vf-u1},
we get
\[\| \chi \partial_\nu w\|_{L^2(\Gamma_2)} \lesssim \| \Lambda^2_{q_2} - \Lambda^2_{q_1} \|^{1/4}_\ast \| v_1 \|^{1/4}_{L^2(\Omega)} \| \chi \partial_\nu w\|^{3/4}_{H^{1/2}(\Gamma_2)}. \]
In order to estimate the last factor on the right-hand side, we are going to use the boundedness of the trace operator in $\Omega$ and the well-posedness of \eqref{pb:w} to get control on $\| w \|_{H^2(\Omega)}$. Thus, we get
\[ \| \chi \partial_\nu w\|_{H^{1/2}(\Gamma_2)} \lesssim \| w \|_{H^2(\Omega)} \lesssim \| v_1 \|_{L^2(\Omega)}, \]
which implies
\begin{equation}
\| \chi \partial_\nu w\|_{L^2(\Gamma_2)} \lesssim \| \Lambda^2_{q_2} - \Lambda^2_{q_1} \|^{1/4}_\ast \| v_1 \|_{L^2(\Omega)}.
\label{es:boundaryL2}
\end{equation}

Finally, gathering \eqref{es:boundaryCARLE}, \eqref{es:otraantes},
\eqref{es:unamas} and the computations to estimate each term on \eqref{term:outputCARLE}, we can state that
\begin{equation}
\begin{aligned}
\bigg| \int_{\Gamma^N_1} \chi u_2 \partial_\nu w\, d x' \bigg| \lesssim \| e^{\tau x \cdot \zeta} u_2 \|_{H^1(\Omega)} \Bigg[ \frac{1}{\tau^{1/2}} \| e^{-\tau x \cdot \zeta} u_1 \|_{L^2(\Omega)} \\
+ \frac{e^{c\tau|\zeta|}}{\tau^{1/2}} \big( \| w \|_{H^1(Q')} + \| v_1 - u_1 \|_{L^2(\Omega)} \big) \\
+ |\zeta|^{1/2} e^{c\tau |\zeta|} \| \Lambda^2_{q_2} - \Lambda^2_{q_1} \|^{1/4}_\ast \| v_1 \|_{L^2(\Omega)} \Bigg].
\end{aligned} \label{es:BOUNDARYtermCARLE}
\end{equation}

Before proceeding with the proof of the claimed integral estimates, let us write down what the estimate, under the assumptions in (b), looks like at this stage: by \eqref{es:INTestimate}, \eqref{ter:boundaryDIRECTLY} and \eqref{es:BOUNDARYtermCARLE}, we obtain
\begin{equation}
\begin{aligned}
\bigg|\int_\Omega (q_1 &- q_2) u_1 u_2 \, dx\bigg| \lesssim \| u_1 - v_1 \|_{L^2(\Omega)} \| u_2 \|_{L^2(\Omega)} \\
 &+ \| w \|_{L^2(Q')} \| u_2 \|_{H^1(\Omega)} + \| \Lambda_{q_1}^2 - \Lambda_{q_2}^2 \|_\ast \| u_2 \|_{H^2 (\Omega)} \| v_1 \|_{L^2(\Omega)}\\
 & + \| e^{\tau x \cdot \zeta} u_2 \|_{H^1(\Omega)} \Bigg[ \frac{1}{\tau^{1/2}} \| e^{-\tau x \cdot \zeta} u_1 \|_{L^2(\Omega)} \\
&+ \frac{e^{c\tau|\zeta|}}{\tau^{1/2}} \big( \| w \|_{H^1(Q')} + \| v_1 - u_1 \|_{L^2(\Omega)} \big) \\
&+ |\zeta|^{1/2} e^{c\tau |\zeta|} \| \Lambda^2_{q_2} - \Lambda^2_{q_1} \|^{1/4}_\ast \| v_1 \|_{L^2(\Omega)} \Bigg]
\end{aligned} \label{es:preUC}
\end{equation}
for all $\tau \geq C (k^2 + M)$ and $\zeta \in \R^3$ with $\zeta \cdot e_3 \geq 1$. 

In the next step, we will control $w$ in $Q'$ by using quantified unique continuation from the boundary. This will be applied to \eqref{es:NGamma1} and \eqref{es:preUC} to obtain the estimates in (a) and (b), respectively. 

Proceed with the control of $w$ in $Q'$. We may assume $w$ not to vanish identically in $Q'$, otherwise we do not have anything to control. In order to estimate a non-identically-vanishing $w$, we will
apply an estimate due to Phung (see Th\'eor\`eme 1.1 in \cite{P}) which reads as follows in our particular case: Let $U$ be a smooth open subset of $\Omega$ containing $Q'$ with $U \cap Q = \emptyset$. Then, there exists a $d > 0$, which depends on $U$, $\Gamma$ and $k$, such that, if
\begin{equation}
\frac{\|w\|_{H^2(U)}}{\| \partial_\nu w \|_{L^2(\Gamma)}} \geq \frac{1}{d}, \label{ineq:a-prioriCOND}
\end{equation}
with $\Gamma := \{ x \in \Gamma^N_l : R + \epsilon < |x'| < R' - \epsilon \}$ for the $\epsilon$ already chosen, then
\begin{equation}
\| w \|_{H^1(U)} \lesssim \frac{\| w \|_{H^2(U)} } {\left[ \log \left(e \frac{ d \| w \|_{H^2(U)}}{\| \partial_\nu w \|_{L^2(\Gamma)}} \right) \right]^{1/2}}. \label{ineq:uc}
\end{equation}
Obviously, $\Gamma \subseteq \partial U \cap \Gamma_l$. Note that, by $w|_{\partial \Sigma} = 0$ and by unique continuation from the boundary, we can ensure that $\| \partial_\nu w \|_{L^2(\Gamma)} > 0$.

On the one hand, by the well-posedness of the problem satisfied by $w$, we know that
\[\| w \|_{H^2(\Omega)} \lesssim \| v_1 - u_1 \|_{L^2(\Omega)} + \| u_1 \|_{L^2(\Omega)}\]
with the implicit constant depending on $\Omega, M$ and $k$. On the other hand, considering another bump function $\chi'$ in $\R^2$ such that $\chi(x') = 1$ for $|x'| \leq R' - \epsilon$ and $\chi'(x') = 0$ for $|x'| > R' - \epsilon/2$, we have, by the same argument that we used to get \eqref{es:boundaryL2} with $\chi'$ instead of $\chi$, that
\begin{equation}
\begin{aligned}
\| \partial_\nu w\|_{L^2(\Gamma)} &\leq \| \chi' \partial_\nu w\|_{L^2(\Gamma_l)} \\
& \lesssim \| \Lambda^l_{q_2} - \Lambda^l_{q_1} \|^{1/4}_\ast \big(\| u_1 \|_{L^2(\Omega)}
+ \| v_1 - u_1 \|_{L^2(\Omega)} \big).
\end{aligned} \nn 
\end{equation}
Obviously, the implicit constants in the previous inequalities can be chosen to be the same.
Thus, since the function
\[ t \mapsto \frac{t}{(\log t)^{1/2}} \]
is increasing on $(e, \infty)$ and since the right-hand side of \eqref{ineq:uc} can be written as
\[ \frac{\| \partial_\nu w \|_{L^2(\Gamma)}}{e d} \frac{e \frac{ d \| w \|_{H^2(U)}}{\| \partial_\nu w \|_{L^2(\Gamma)}}} { \left[ \log \left(e \frac{ d \| w \|_{H^2(U)}}{\| \partial_\nu w \|_{L^2(\Gamma)}} \right) \right]^{1/2} }, \]
the last two inequalities can be combined with \eqref{ineq:a-prioriCOND} and
\eqref{ineq:uc} to deduce the following: if $\| \Lambda^l_{q_2} - \Lambda^l_{q_1} \|_\ast < d^4$, we have
\begin{equation}
\| w \|_{H^1(U)} \lesssim \frac{\| v_1 - u_1 \|_{L^2(\Omega)} + \| u_1 \|_{L^2(\Omega)} } {\left[ 1 + \bigg| \log \left( d^{-4} \| \Lambda_{q_2}^l - \Lambda_{q_1}^l \|_\ast \right) \bigg| \right]^{1/2}}. \label{es:locUC} 
\end{equation}
From now until the end of the proof, we shall write $\delta := d^{-4}$ and we shall assume $\| \Lambda^l_{q_2} - \Lambda^l_{q_1} \|_\ast < \delta^{-1}$ (so that we do not have to state this condition explicitly every time).

At this stage, the proofs of both parts of Proposition \ref{prop:integralEst} are almost complete. What remains for us to do is, firstly, to apply \eqref{es:locUC} to each inequality of \eqref{es:NGamma1}, \eqref{es:preUC} thus obtaining two new inequalities and, secondly, to apply the announced Runge-type approximation to the two new inequalities. We now go on to finish the proof of Proposition \ref{prop:integralEst}, whereby we shall omit all lengthy but straightforward calculations.

The Runge-type approximation can be stated as follows: For all $u_1$ as in the statement of Proposition \ref{prop:integralEst} and $\varepsilon > 0$, there exists a $v_1 \in H^2_\mathrm{loc} (\overline{\Sigma}) $ solving $ (-\Delta - k^2 + q_1) v_1 = 0 $ in $ \Sigma $ with $ \supp (v_1|_{\partial \Sigma}) \subseteq \overline{\Gamma^D_1}$ such that
\[\| v_1 - u_1 \|_{L^2(\Omega)} < \varepsilon.\]

With regard to part (a) of Proposition \ref{prop:integralEst}: by applying \eqref{es:locUC} to \eqref{es:NGamma1} and then by applying the approximation result to the resulting inequality, we obtain
\begin{equation}
\begin{aligned}
\bigg|\int_\Omega (q_1 - q_2) u_1 u_2 \, dx\bigg| \lesssim & \frac{  \| u_1 \|_{L^2(\Omega)} \| u_2 \|_{H^1(\Omega)} } {\left[ 1 + \bigg| \log \left(\delta \| \Lambda^1_{q_2} - \Lambda^1_{q_1} \|_\ast \right) \bigg| \right]^{1/2}} \\
&+   \| \Lambda_{q_1}^1 - \Lambda_{q_2}^1 \|_\ast \| u_1 \|_{L^2(\Omega)} \| u_2 \|_{H^2 (\Omega)}.
\end{aligned} \label{es:NGamma1'}
\end{equation}
With regard to part (b) of Proposition \ref{prop:integralEst}: we argue analogously by firstly applying \eqref{es:locUC} to \eqref{es:preUC} and secondly by applying the approximation result to obtain
\begin{equation}
\begin{aligned}
\bigg|&\int_\Omega (q_1 - q_2) u_1 u_2 \, dx\bigg| \lesssim \frac{  \| u_1 \|_{L^2(\Omega)} \| u_2 \|_{H^1(\Omega)} } {\left[ 1 + \bigg| \log \left(\delta \| \Lambda^2_{q_2} - \Lambda^2_{q_1} \|_\ast \right) \bigg| \right]^{1/2}}  \\
 &+ \| \Lambda_{q_1}^2 - \Lambda_{q_2}^2 \|_\ast \| u_1 \|_{L^2(\Omega)} \| u_2 \|_{H^2 (\Omega)} \\
 &+ \| e^{\tau x \cdot \zeta} u_2 \|_{H^1(\Omega)}  \Bigg( \frac{e^{c\tau|\zeta|}}{\tau^{1/2}} \frac{  \| u_1 \|_{L^2(\Omega)} } {\left[ 1 + \bigg| \log \left(\delta \| \Lambda^2_{q_2} - \Lambda^2_{q_1} \|_\ast \right) \bigg| \right]^{1/2}} \\
 &+ \frac{1}{\tau^{1/2}} \| e^{-\tau x \cdot \zeta} u_1 \|_{L^2(\Omega)} + |\zeta|^{1/2} e^{c\tau |\zeta|} \| \Lambda^2_{q_2} - \Lambda^2_{q_1} \|^{1/4}_\ast \| u_1 \|_{L^2(\Omega)} \Bigg).
\end{aligned} \label{es:NGamma2'}
\end{equation}
By dropping higher-order terms from the right-hand side of \eqref{es:NGamma1'} and \eqref{es:NGamma2'} (possibly at the cost of increasing the implicit constants in each of these inequalities), we arrive at the estimate claimed in (a) and (b).
%
\end{proof}

\section{Proof of Theorem \ref{th:diffDdNd}} \label{sec:theorem1}
In this section, we prove Theorem \ref{th:diffDdNd}. To achieve this task, we will construct appropriate CGOs, use those CGOs to construct the functions $u_1$ and $u_2$ appearing in the integral estimate of Proposition \ref{prop:integralEst} (b), and eventually obtain an upper bound on $(\widehat{q_1} - \widehat{q_2})(\xi)$ at each frequency $\xi$ from
\[\{ \xi = (\xi',\xi_3) \in \R^3: 1 \leq |\xi'| < r, |\xi_3| < r \}; \]
then, we will extend our control on $\widehat{q_1} - \widehat{q_2}$ to the ball
\[\{ \xi \in \R^3: |\xi| < r \}. \]
After this, we will carry out a classical argument due to Alessandrini (see \cite{A}) in order to obtain the stability estimate.

From now until the end of this section, we abuse notation by letting $q_j$ stand both for the potential from the statement of Theorem \ref{th:diffDdNd} (which is only defined on $\Sigma$) and for its trivial extension to all of $\R^3$. The meaning will be clear from the context; for example, $\widehat{q_j}$ refers to the Fourier transform of the trivial extension of $q_j$ to all of $\R^3$.

Start by stating the CGOs used to prove Theorem \ref{th:diffDdNd}. We perform the reflection argument originating from the work of Isakov in \cite{I}. Let $r > 2$, which will be specified later on in this section.

Let $\xi \in \R^3$ with 
\begin{equation}
1 \leq \xi_{1e} := \sqrt{\xi_1^2 + \xi_2^2} < r \quad \text{and} \quad |\xi_3|< r, \label{eq:restrictions on xi}
\end{equation}
be arbitrarily chosen. We define
\begin{align*}
e(1) &:= \frac{1}{\xi_{1e}}(\xi_1, \xi_2, 0), \\
e(3) &:= (0,0,1), \\
e(2) &:= e(3) \times e(1) = \frac{1}{\xi_{1e}}(-\xi_2, \xi_1, 0).
\end{align*}
We set $x^* := (x_1, x_2, -x_3)$ for any $x = (x_1, x_2, x_3)  \in \R^3$, $f^*(x):= f(x^*)$  for any function $f$, and $G^\ast = \{ x^\ast : x \in G \}$ for any domain $G$. The coordinates of any $x \in \R^3$ with respect to the orthonormal basis $\{ e(j) \}_{j=1}^3$ shall be denoted by $x = (x_{1e}, x_{2e}, x_{3e})_e$. Note $\xi = (\xi_{1e},0,\xi_3)_e$. We also write $\xi^\perp := (-\xi_3, 0, \xi_{1e})_e$.

As preparation for the reflection argument, we now fix a smooth bounded domain $B \subseteq \R^3$ such that
\[
\overline{\Omega \cup \Omega^*} \subseteq B, \quad B^* = B.
\]
Let $Q_1 \in L^\infty(B)$ be the even extension of $q_1$ about the coordinate variable $x_3$ and $Q_2 \in L^\infty(B)$ be the trivial extension of $q_2$ to all of $B$; explicitly, we define
\begin{eqnarray}
Q_1(x) &:=& q_1 (x) \chi_\Sigma(x) + q_1(x^*) \chi_{\Sigma^*}(x), \nn \\
Q_2(x) &:=& q_2 (x) \chi_\Sigma(x), \nn
\end{eqnarray}
for a.e. $x \in \R^3$, where $\chi_\Sigma$ and $\chi_{\Sigma^*}$ denote the characteristic functions of $\Sigma $ and $\Sigma^*$ respectively.

As in \cite{LU}, we introduce
\begin{eqnarray}
\rho_1 &:=& \left( -\tau \xi_3 + \frac{i}{2}\xi_{1e},i |\xi| (\tau^2-1/4)^{1/2}, \tau \xi_{1e} + \frac{i}{2} \xi_3 \right)_e \nn \\
&=& \tau \xi^\perp + i \left( \frac{1}{2} \xi + |\xi| (\tau^2 - 1/4)^{1/2} e(2) \right), \nn \\
\rho_2 &:=& \left( \tau \xi_3 + \frac{i}{2} \xi_{1e}, - i |\xi| (\tau^2 - 1/4)^{1/2}, -\tau \xi_{1e} + \frac{i}{2} \xi_3 \right)_e \nn \\
&=& - \tau \xi^\perp + i \left( \frac{1}{2} \xi - |\xi| (\tau^2 - 1/4)^{1/2} e(2) \right). \nn
\end{eqnarray}
One immediately computes that
\begin{equation}
\rho_m \cdot \rho_m = 0, \quad |\rho_m| = \sqrt{2}\tau |\xi|, \quad m = 1, 2. \label{eq:complex vectors}
\end{equation}
The $\rho_1$ and $\rho_2$ will be the candidates to construct the family of CGOs. It is a well-known fact that there exists a function $V_m \in H^2(B)$ solving
\begin{equation}
(-\Delta + Q_m - k^2) V_m = 0 \text{ in } B
\end{equation}
and having the form $V_m = e^{x \cdot \rho_m} (1+ \psi_m)$, where the remainder $\psi_m$ obeys
\begin{equation}
\| \psi_m \|_{H^k(B)} \lesssim \frac{1}{\tau^{1-k}}, \quad k=0,1,2, \label{ineq:decay conditions for remainders}
\end{equation}
for all $\tau \geq \tau_1:= \max(C_0 (M + k^2), 1)$, with $C_0 \geq 1$ depending on $B$. The implicit constant in \eqref{ineq:decay conditions for remainders} depends on $B$, $M$ and $k$.

Recall that Proposition \ref{prop:integralEst} (b) requires for $u_1$ to satisfy $u_1|_{\partial \Omega \cap \Gamma_2} = 0$; this boundary condition can be arranged to hold via Isakov's reflection argument from \cite{I}. Employing the same idea as in \cite{LU}, we set
\begin{eqnarray}
u_1(x)&:=& e^{x\cdot \rho_1}( 1+\psi_1(x)) - e^{x^* \cdot \rho_1} (1 + \psi_1^*(x)), \label{eq:def of u_1} \\
u_2(x)&:=& e^{x\cdot \rho_2}( 1+\psi_2(x)). \label{eq:def of u_2}
\end{eqnarray}
The construction of $\psi_1, \psi_2, u_1, u_2$ ensures that $u_1|_\Omega$ and $u_2|_{\Omega}$ satisfy the hypotheses of Proposition \ref{prop:integralEst} (b).

Let us compute that
\begin{align*}
\int_\Sigma (q_1 - q_2) u_1 u_2\, dx =& \int_\Sigma e^{i x \cdot \xi} (1+ \psi_1) (1+\psi_2) (q_1 - q_2)\, dx \\
&- \int_\Sigma (q_1 - q_2) e^{i x_{1e} \xi_{1e}} e^{-2\tau x_3 \xi_{1e}} ( 1 + \psi_1^*)(1+\psi_2)\, dx.
\end{align*}
As a consequence, we obtain
\begin{equation}
\begin{aligned}
&\left| \int_\Sigma e^{ix\cdot \xi} (q_1 - q_2)\, dx \right| \leq \left| \int_\Sigma (q_1 - q_2) u_1 u_2\, dx \right|  \\
&+ \left| \int_\Sigma e^{ix\cdot \xi} (q_1 -q_2) (\psi_1 + \psi_2 + \psi_1 \psi_2) \, dx \right|  \\
&+ \left| \int_\Sigma (q_1 -q_2) e^{i x_{1e} \xi_{1e}} e^{-2\tau x_3 \xi_{1e}} ( 1 + \psi_1^*)(1+\psi_2) \, dx \right|. 
\end{aligned} \label{eq:estimate after plugging in CGOs}
\end{equation}
By applying the triangle inequality, using that supp$(q_m) \subseteq Q $, and using $\| q_m \|_{L^\infty(\Sigma)} \lesssim 1$, we verify that
\begin{equation}
\left| \int_\Sigma (q_1-q_2)e^{ix_{1e} \xi_{1e}} e^{-2\tau x_3 \xi_{1e}} \, dx \right| \lesssim \frac{1}{ \tau}. \label{eq:a more precise version of Lemma lemma:bounding the last term in the PDF from 23 Septmeber 2014}
\end{equation}

Let us now apply \eqref{ineq:decay conditions for remainders} and \eqref{eq:a more precise version of Lemma lemma:bounding the last term in the PDF from 23 Septmeber 2014} to \eqref{eq:estimate after plugging in CGOs} to obtain
\begin{equation*}
\left| \int_\Sigma e^{ix\cdot \xi} (q_1 - q_2)\, dx \right| \lesssim \left| \int_\Sigma (q_1 - q_2) u_1 u_2\, dx \right| + \frac{1}{\tau}
\end{equation*}
for $\tau \geq \tau_1$, and $u_1, u_2$ defined by \eqref{eq:def of u_1} and \eqref{eq:def of u_2}. 
As noted earlier, the functions $u_1|_\Omega$ and $u_2|_{\Omega}$ satisfy the hypotheses of Proposition \ref{prop:integralEst}, so we may apply Proposition \ref{prop:integralEst} (b) with $ \zeta = \xi^\perp$ to deduce that, if $\| \Lambda^2_{q_1} - \Lambda^2_{q_2}\|_\ast < 1/\delta$, then
\begin{equation*}
\begin{aligned}
\bigg|\int_\Sigma (q_1 - q_2) e^{ix\cdot \xi} \, dx\bigg| \lesssim & \frac{1}{\tau} + e^{c \tau |\xi| } \frac{\| u_1 \|_{L^2(\Omega)} \| u_2 \|_{H^2 (\Omega)}} {\left[ 1 + \bigg| \log \left(\delta \| \Lambda^2_{q_2} - \Lambda^2_{q_1} \|_\ast \right) \bigg| \right]^{1/2}} \\
&+ \frac{1}{\tau^{1/2}} \| e^{\tau x \cdot \xi^\perp} u_2 \|_{H^1 (\Omega)} \| e^{-\tau x \cdot \xi^\perp} u_1 \|_{L^2(\Omega)}
\end{aligned}
\end{equation*}
for all $\tau \geq \max(\tau_0, \tau_1)$.

The choices of $\rho_m$ and $u_m$ can be combined with \eqref{ineq:decay conditions for remainders} to deduce that
\begin{equation*}
\bigg|\int_\Sigma (q_1 - q_2) e^{ix\cdot \xi} \, dx\bigg| \lesssim \frac{e^{c \tau |\xi| }} {\left[ 1 + \bigg| \log \left(\delta \| \Lambda^2_{q_2} - \Lambda^2_{q_1} \|_\ast \right) \bigg| \right]^{1/2}} + \frac{1}{\tau^{1/2}}
\end{equation*}
for all $\tau \geq \max(\tau_0, \tau_1)$, with $c > 4(2R + L)$. Thus, we obtain the uniform estimate
\begin{equation}
\big| \widehat{q_1}(\xi) - \widehat{q_2} (\xi) \big| \lesssim \frac{e^{c \tau r }} {\left[ 1 + \bigg| \log \left(\delta \| \Lambda^2_{q_2} - \Lambda^2_{q_1} \|_\ast \right) \bigg| \right]^{1/2}} + \frac{1}{\tau^{1/2}}. \label{eq:est for large freqs}
\end{equation}
for all $\tau \geq \max(\tau_0, \tau_1)$ and all $\xi \in \R^3$ with $1 \leq \xi_{1e} < r$, $|\xi_3| < r$.
Now, we are going to use analytic continuation in order to extend the set of frequencies, at which we control the difference $\widehat{q_1} - \widehat{q_2}$, to all of $\{ |\xi| < r \}$.

Let $\xi \in \R^3$ with $0<\xi_{1e} < 1$, $|\xi_3| < r$ be arbitrarily chosen; define $e(1), e(2), e(3)$ as we did earlier. By the Payley-Wiener theorem, $\widehat{q_1} - \widehat{q_2}$ is the restriction to $\R^3$ of an entire function on $\C^3$. Therefore, the function $f$ defined by
\begin{eqnarray}
f: \C &\rightarrow& \C \nn \\
	z &\mapsto& (\widehat{q_1} - \widehat{q_2})\left( (z, 0, \xi_3)_e \right) \nn
\end{eqnarray}
is entire. If we define
\begin{align*}
G := \{s + it \in \mathbb{C} : |s| < 2, |t| < 2 \}, \\
\gamma := \{ s + it \in \mathbb{C} : 0 < s < 1, t = 0 \},\\
\Gamma_0 := \{ s + it \in \mathbb{C} : 1 < s < 2, t = 0 \},
\end{align*}
then Corollary 1.2.2 (b) from \cite{Ibook} implies that there exist constants $C_0> 0$ and $\lambda \in (0,1)$, both of which depend on $\gamma$, such that
\[\sup_{\gamma} |f(s)| \leq C_0 (\sup_{G} |f (s +i t)|)^{1 - \lambda} (\sup_{\Gamma_0} |f (s)|)^\lambda. \]
Since
\[ \sup_G |f(s + it)| \lesssim 1,\]
and since $\sup_{\Gamma_0} |f (s)|$ can be bounded by means of \eqref{eq:est for large freqs}, we can conclude
\begin{equation}
\big| \widehat{q_1}(\xi) - \widehat{q_2} (\xi) \big| \lesssim \frac{e^{c \lambda \tau r }} {\left[ 1 + \bigg| \log \left(\delta \| \Lambda^2_{q_2} - \Lambda^2_{q_1} \|_\ast \right) \bigg| \right]^{\lambda/2}} + \frac{1}{\tau^{\lambda/2}} \label{eq:est for small freqs}
\end{equation}
for all $\tau \geq \max(\tau_0, \tau_1)$ and $\xi \in \R^3$ with $0<\xi_{1e} < 1$, $|\xi_3| < r$.

We go on to combine \eqref{eq:est for large freqs} and \eqref{eq:est for small freqs}, then drop higher-order terms (possibly at the cost of increasing the implicit constant), and thus conclude the following:
\begin{equation}
\big| \widehat{q_1}(\xi) - \widehat{q_2} (\xi) \big| \lesssim \frac{e^{c \tau r }} {\left[ 1 + \bigg| \log \left(\delta \| \Lambda^2_{q_2} - \Lambda^2_{q_1} \|_\ast \right) \bigg| \right]^{\lambda/2}} + \frac{1}{\tau^{\lambda/2}} \label{es:fourier<r red}
\end{equation}
for all $\tau \geq \max(\tau_0, \tau_1)$ and $\xi \in \R^3$ with $|\xi| < r$.

Next, we finish the proof of Theorem \ref{th:diffDdNd} by performing the classical argument due to Alessandrini \cite{A}. 
If we put $\varepsilon := \frac{s-\frac{3}{2}}{2}$ (so that $s = \frac{3}{2} + 2\varepsilon$), we may apply the Sobolev embedding theorem and interpolation together with the a-priori bounds on $q_1, q_2$ to obtain
\begin{equation}
\begin{aligned}
 \| q_1 - q_2 \|_{L^\infty(\Sigma)} &=  \| q_1 - q_2 \|_{L^\infty(\Omega)} \lesssim \| q_1 - q_2 \|_{H^{\frac{3}{2}+\varepsilon}(\Omega)} \\
	&\leq \| q_1 - q_2 \|_{H^{-1}(\Omega)}^{ \frac{\varepsilon}{s+1} } \| q_1 - q_2 \|_{H^s(\Omega)}^{ \frac{s-\varepsilon+1}{s+1}  } \\
	&\lesssim \| q_1 - q_2 \|_{H^{-1}(\Omega)}^{ \frac{\varepsilon}{s+1} } \leq \| q_1 - q_2 \|_{H^{-1}(\R^3)}^{ \frac{\varepsilon}{s+1} }.
\end{aligned} \label{ienq:fifth estimate}
\end{equation}
On the other hand, by using the definition of $\| \centerdot \|_{H^{-1}(\R^3)}$ in terms of the Fourier transform, then splitting the integral into high and low frequencies, and lastly using Plancharel's theorem, we get
\[\| q_1 - q_2 \|_{H^{-1}(\R^3)}^2 \lesssim r^3 \sup_{\{|\xi| < r\}} | \widehat{q_1}(\xi) - \widehat{q_2} (\xi) |^2 + r^{-2}. \]
Applying 
\eqref{es:fourier<r red} to the last estimate, utilizing $\tau \geq 1$, and for $c > 4(2R + L)+1$, we get
\[\| q_1 - q_2 \|_{H^{-1}(\R^3)} \lesssim \frac{e^{c \tau r }} {\left[ 1 + \bigg| \log \left(\delta \| \Lambda^2_{q_2} - \Lambda^2_{q_1} \|_\ast \right) \bigg| \right]^{\lambda/2}} + r^{3/2} \tau^{-\lambda/2} + r^{-1}. \]
Upon selecting $\tau$ so that $r^{-1} = r^{3/2} \tau^{-\lambda/2}$ or, equivalently, as $\tau:=r^{5/\lambda}$, the preceding estimate implies
\[\| q_1 - q_2 \|_{H^{-1}(\R^3)} \lesssim \frac{e^{c r^{\frac{\lambda + 5}{\lambda} } }} {\left[ 1 + \bigg| \log \left(\delta \| \Lambda^2_{q_2} - \Lambda^2_{q_1} \|_\ast \right) \bigg| \right]^{\lambda/2}} + r^{-1}. \]
Choose $r>0$ so that
\[
r^{\frac{\lambda+5}{\lambda}} = c^{-1} \log \left\{ \left[1 + \left| \log( \delta \| \Lambda_{q_2}^2 - \Lambda_{q_1}^2 \|_\ast ) \right| \right]^{\lambda/4}  \right\}
\]
in the last inequality and combine it with \eqref{ienq:fifth estimate};  in the resulting inequality, we drop higher-order terms (possibly at the cost of increasing the implicit constant), and thus derive the stability estimate of Theorem \ref{th:diffDdNd} with $\theta := \frac{\lambda}{2(\lambda+5)}$.

\section{Proof of Theorem \ref{th:sameDdNd}} \label{sec:theorem2}
In this section, we prove Theorem \ref{th:sameDdNd}. In doing so, we imitate the arguments from Section \ref{sec:theorem1}; broadly speaking, the main difference is that occurrences of $(\widehat{q_1} - \widehat{q_2})$ from Section \ref{sec:theorem1} will now be replaced by occurrences of $( Q_1^\text{even} - Q_2^\text{even})\widehat{\ }$, where $Q_j^\text{even}$ stands for the even extension of $q_j|_{\{ x_3 \geq 0 \}}$ to $\R^3$ about the coordinate variable $x_3$.

As in Section \ref{sec:theorem1}, we begin by constructing appropriate CGOs by means of Isakov's reflection argument from \cite{I}. Consider $r>2$, which will be specified later on in this section.

Let $\xi \in \R^3$ with
\[
1 \leq \xi_{1e} < r \quad \text{and} \quad |\xi_3| < r,
\]
be arbitrarily chosen. Define $e(1), e(2), e(3)$ as in Section $\ref{sec:theorem1}$.

From now until the end of this section, we let $Q_j^\text{even}$ stand for the even extension of $q_j$ about the coordinate variable $x_3$; explicitly, we set
{\[
Q_j^\text{even}(x) := q_j(x) + q_j(x^*) \qquad \text{for a.e. } x \in \R^3.
\]
Thanks to the regularity hypotheses on $q_j$, we have that $Q_1^\text{even}$ and $ Q_2^\text{even}$ belong to $H^s(\R^3)$ and have their supports contained in $Q \cup Q^*$.

Fix $B$ as in Section $\ref{sec:theorem1}$. Following the idea from Section 4 in \cite{LU}, we will construct $u_1$ and $u_2$ via Isakov's reflection argument. Firstly, we define
\begin{eqnarray}
\rho_1 &:=& \left( i \left( \frac{ \xi_{1e} }{2} - \alpha \xi_{3}\right) , - (\alpha^2+1/4)^{1/2} |\xi|,  i \left( \frac{ \xi_{3} }{2} + \alpha \xi_{1e}\right) \right)_e, \label{eq:first phase in second case}  \\
\rho_2 &:=& \left( i \left( \frac{ \xi_{1e} }{2} + \alpha \xi_{3}\right) , (\alpha^2+1/4)^{1/2} |\xi|,  i \left( \frac{ \xi_{3} }{2} - \alpha \xi_{1e}\right) \right)_e, \label{eq:second phase in second case}
\end{eqnarray}
where $\alpha > 0$ is a parameter. One readily verifies that
\[
\rho_m \cdot \rho_m = 0, \quad |\rho_m| = \sqrt{2}|\xi| (\alpha^2 + 1/4)^{1/2}, \quad m=1,2.
\]
It is a well-known fact that there exists a constant $C_0 = C_0(B, M, k) \geq 1$ such that, for each $\alpha \geq \alpha_2:= \max(C_0(M+k^2), 1)$, there exists a function $\psi_m \in H^2(B)$ satisfying
\begin{equation}
\| \psi_m \|_{H^k(B)} \lesssim \frac{1}{\big[ (\alpha^2 + 1/4)^{1/2} |\xi| \big]^{1-k}}, \quad k=0,1,2, \quad m =1,2 \label{eq:decay in second case}
\end{equation}
such that $V_m(x):= e^{x\cdot \rho_m}(1+ \psi_m)$ belongs to $H^2(B)$ and satisfies
\[
( -\Delta + Q_m^{ \text{even} } - k^2) V_m = 0 \text{ in } B.
\]
The implicit constant in \eqref{eq:decay in second case} depends on $B, M, k$. Regarding the existence of CGOs, in this context, we refer to \cite{BU}, \cite{DFKSU} and \cite{HUW}.

Employing the same idea as in \cite{LU}, we set
\begin{equation}
u_m(x) := e^{x \cdot \rho_m} (1+ \psi_m) - e^{x^* \cdot \rho_m} (1+\psi_m^*); \label{eq:explicit functions in second case}
\end{equation}
it then follows that $u_1|_\Omega$ and $u_2|_\Omega$ satisfy the hypotheses of Proposition \ref{prop:integralEst} (a). On the one hand, a routine computation utilizing the decay estimates from \eqref{eq:decay in second case} shows that
\begin{eqnarray}
\| u_1 \|_{L^2(\Omega)} &\lesssim& e^{ cr (\alpha^2+1/4)^{1/2} }, \label{eq:routineEst1 in second case} \\
\| u_2 \|_{H^2(\Omega)} &\lesssim& e^{ c r (\alpha^2+1/4)^{1/2} }, \label{eq:routineEst2 in second case}
\end{eqnarray}
with the implicit constants depending on $B, n, M, k$; at this stage, we have increased $c$ if necessary. On the other hand, a direct calculation shows that
\begin{equation}
\begin{aligned}
\int_\Sigma &(q_1 - q_2) u_1 u_2 \, dx =\\
& \int_\Sigma (q_1 - q_2) e^{i x \cdot \xi}\, dx + \int_\Sigma (q_1 - q_2) e^{ix \cdot \xi}(\psi_1 +\psi_2 +\psi_1 \psi_2)\, dx \\
&- \int_\Sigma (q_1 - q_2) e^{ix \cdot (\xi_{1e},0,2\alpha \xi_{1e})_e}\, dx\\
&- \int_\Sigma (q_1 - q_2) e^{i x\cdot (\xi_{1e},0,2\alpha \xi_{1e})_e} (\psi_1 +\psi_2^* +\psi_1 \psi_2^*)\, dx \\
&- \int_\Sigma (q_1 - q_2) e^{ix \cdot (\xi_{1e},0,-2\alpha \xi_{1e})_e}\, dx \\
& - \int_\Sigma (q_1 - q_2) e^{i x\cdot (\xi_{1e},0,-2\alpha \xi_{1e})_e} (\psi_1^* +\psi_2 +\psi_1^* \psi_2)\, dx \\
&+ \int_\Sigma (q_1 - q_2) e^{i x^* \cdot \xi}\, dx +  \int_\Sigma (q_1 - q_2) e^{i x^* \cdot \xi}(\psi_1^* +\psi_2^* +\psi_1^* \psi_2^*)\, dx;
\end{aligned} \nn
\end{equation}
combining this with the hypotheses on $q_1 - q_2$, \eqref{eq:decay in second case}, and $|\xi| \geq \xi_{1e} \geq 1$ establishes
\begin{equation}
\begin{aligned}
&\left| \int_\Sigma (q_1 - q_2) e^{i x \cdot \xi}\, dx + \int_\Sigma (q_1 - q_2) e^{i x^* \cdot \xi}\, dx\right| \lesssim \left| \int_\Sigma (q_1 - q_2) u_1 u_2 \, dx \right| \\
&+ \left| \left( q_1 - q_2 \right)^{\widehat{\ }} \big( (-\xi_{1e},0,-2\alpha \xi_{1e})_e \big) \right| + \left| \left( q_1 - q_2 \right)^{\widehat{\ }} \big( (-\xi_{1e},0,2\alpha \xi_{1e})_e \big) \right| \\
&+ \frac{1}{ (\alpha^2 + 1/4)^{1/2} },
\end{aligned} \nn
\end{equation}
where the implicit constant depends on $B, M, k$. For technical reasons, let us replace $\xi$ by $-\xi$. Now, we apply the quantified Riemann-Lebesgue lemma to $f := q_1 - q_2$ (in order to handle the Fourier transforms on the right-hand side in the last inequality), Proposition \ref{prop:integralEst} (a), \eqref{eq:routineEst1 in second case}, and \eqref{eq:routineEst2 in second case} to obtain
\begin{equation}
| (Q_1^\text{even} - Q_2^\text{even})\widehat{\ }(\xi) | \lesssim \frac{e^{cr(\alpha^2 + 1/4)^{1/2}}}{ \left[ 1 + | \log(\delta \| \Lambda_{q_1}^1 - \Lambda_{q_2}^1 \|_\ast ) | \right]^{1/2} } + \frac{1}{(\alpha^2 + 1/4)^{1/2}} \nn 
\end{equation}
whenever $\alpha \geq \alpha_2$. Here we have increased $c$.

For the sake of brevity and the ease of comparison with the arguments from the previous section, let us introduce a new parameter $\tau := (\alpha^2 + 1/4)^{1/2}$. Using this new parameter, we have obtained the following inequality: there exists a constant $T_2 := (\alpha_2^2 + 1/4)^{1/2}$ such that
\begin{equation}
| (Q_1^\text{even} - Q_2^\text{even})\widehat{\ }(\xi) | \lesssim \frac{e^{cr\tau}}{ \left[ 1 + | \log(\delta \| \Lambda_{q_1}^1 - \Lambda_{q_2}^1 \|_\ast ) | \right]^{1/2} } + \frac{1}{\tau} \label{eq:est for large freqs in second case} 
\end{equation}
for all $\tau \geq T_2$ and all $\xi \in \R^3$ with $1 \leq \xi_{1e} < r$, $|\xi_3| <r$.

Now, we are going to use analytic continuation in order to extend the set of frequencies, at which we control the difference $(Q_1^\text{even} - Q_2^\text{even})\widehat{\ }$, to all of $\{ |\xi| < r\}$. 

Let $\xi \in \R^3$ with $0<\xi_{1e} < 1$, $|\xi_3| < r$ be arbitrarily chosen; define $e(1), e(2), e(3)$ as we did earlier. By the Payley-Wiener theorem, $(Q_1^\text{even} - Q_2^\text{even})\widehat{\ }$ is the restriction to $\R^3$ of an entire function on $\C^3$. Therefore, the function $g$ defined by
\begin{eqnarray}
g: \C &\rightarrow& \C \nn \\
	z &\mapsto& (Q_1^\text{even} - Q_2^\text{even})\widehat{\ }\left( (z, 0, \xi_3)_e \right) \nn
\end{eqnarray}
is entire. If $G, \gamma, \Gamma_0$ stand for the same sets as in Section \ref{sec:theorem1}, then Corollary 1.2.2 (b) from \cite{Ibook} implies that there exist constants $C_0> 0$ and $\lambda \in (0,1)$, both of which depend on $\gamma$, such that
\[\sup_{\gamma} |f(s)| \leq C_0 (\sup_{G} |f (s +i t)|)^{1 - \lambda} (\sup_{\Gamma_0} |f (s)|)^\lambda. \]
Again, as in Section \ref{sec:theorem1}, we verify that $\sup_G |g(s)| \lesssim 1$ while $\sup_{\Gamma_0} |g (s)|$ can be bounded by means of \eqref{eq:est for large freqs in second case}, enabling us to conclude
\begin{equation}
\big| (Q_1^\text{even} - Q_2^\text{even})\widehat{\ }(\xi) \big| \lesssim \frac{e^{c \lambda \tau r }} {\left[ 1 + \bigg| \log \left(\delta \| \Lambda^1_{q_2} - \Lambda^1_{q_1} \|_\ast \right) \bigg| \right]^{\lambda/2}} + \frac{1}{\tau^{\lambda}} \label{eq:est for small freqs in second case}
\end{equation}
for all $\tau \geq T_2$ and $\xi \in \R^3$ with $0<\xi_{1e} < 1$, $|\xi_3| < r$.

We go on to combine \eqref{eq:est for large freqs in second case} and \eqref{eq:est for small freqs in second case}, then drop higher-order terms (possibly at the cost of increasing the implicit constant), and thus conclude the following:
\begin{equation}
\big| (Q_1^\text{even} - Q_2^\text{even})\widehat{\ }(\xi) \big| \lesssim \frac{e^{c \tau r }} {\left[ 1 + \bigg| \log \left(\delta \| \Lambda^1_{q_2} - \Lambda^1_{q_1} \|_\ast \right) \bigg| \right]^{\lambda/2}} + \frac{1}{\tau^{\lambda}} \label{es:fourier<r in second case}
\end{equation}
for all $\tau \geq T_2$ and $\xi \in \R^3$ with $|\xi| < r$.

Next, we finish the proof of Theorem \ref{th:sameDdNd} by performing the classical argument due to Alessandrini \cite{A}. 
If we put $\varepsilon := \frac{s-\frac{3}{2}}{2}$ (so that $s = \frac{3}{2} + 2\varepsilon$), we may apply the Sobolev embedding theorem and interpolation together with the a-priori bounds on $q_1, q_2$ to obtain
\begin{equation}
\begin{aligned}
 \| q_1 &- q_2 \|_{L^\infty(\Sigma)} =  \| Q_1^\text{even} - Q_2^\text{even} \|_{L^\infty(\Omega)}  \\
 	&\lesssim \| Q_1^\text{even} - Q_2^\text{even} \|_{H^{\frac{3}{2}+\varepsilon}(\Omega)} \\
	&\leq \| Q_1^\text{even} - Q_2^\text{even} \|_{H^{-1}(\Omega)}^{ \frac{\varepsilon}{s+1} } \| Q_1^\text{even} - Q_2^\text{even} \|_{H^s(\Omega)}^{ \frac{s-\varepsilon+1}{s+1}  } \\
	&\lesssim \| Q_1^\text{even} - Q_2^\text{even} \|_{H^{-1}(\Omega)}^{ \frac{\varepsilon}{s+1} } \leq \| Q_1^\text{even} - Q_2^\text{even} \|_{H^{-1}(\R^3)}^{ \frac{\varepsilon}{s+1} }.
\end{aligned} \label{ienq:fifth estimate in second case}
\end{equation}
Again, as in Section \ref{sec:theorem1}, by using the definition of $\| \centerdot \|_{H^{-1}(\R^3)}$ in terms of the Fourier transform, then splitting the integral into high and low frequencies, and lastly using Plancharel's theorem, we get
\[
\| Q_1^\text{even} - Q_2^\text{even} \|_{H^{-1}(\R^3)}^2 \lesssim r^3 \sup_{\{|\xi| < r\}} | (Q_1^\text{even} - Q_2^\text{even})\widehat{\ } (\xi) |^2 + r^{-2}.
\]
We proceed by imitating the argument from Section \ref{sec:theorem1}: apply \eqref{es:fourier<r in second case} to the last inequality, insert the resulting inequaliting into \eqref{ienq:fifth estimate in second case}, then select $\tau:=r^{5/(2\lambda)}$, next select $r$ such that
\[
r^{ \frac{2\lambda+5}{2\lambda} } = c^{-1} \log \left\{ \left[1 + \left| \log( \delta \| \Lambda_{q_2}^1 - \Lambda_{q_1}^1 \|_\ast ) \right| \right]^{\lambda/4}  \right\},
\]
drop higher-order terms (possibly at the cost of increasing the implicit constant), and thus derive the stability estimate of Theorem \ref{th:sameDdNd} with $\theta := \frac{\lambda}{2\lambda+5}$.

\begin{acknowledgements} \rm During the preparation of this paper P.C. was part of the University of Helsinki and was supported by the projects ERC-2010 Advanced Grant, 267700 - InvProb and Academy of Finland (Decision number 250215, the Centre of Excellence in Inverse Problems). K.M. is partly supported by the NSF. The authors would like to thank Gunther Uhlmann, who has vigorously facilitated this collaboration, and the organizers of the program on Inverse Problems held in Mittag-Leffler Institut in 2013, where this project was started.
\end{acknowledgements}

\end{document}